\documentclass{article}
\usepackage{amsmath,amssymb,amsthm,geometry,setspace,bm,array}
\usepackage[{colorlinks,linkcolor=blue,anchorcolor=blue,citecolor=blue}]{hyperref}

\def\R{\mathbb{R}}
\def\Z{\mathbb{Z}}
\def\Q{\mathbb{Q}}
\def\P{\mathbb{P}}
\def\rep{\text{representation}}
\def\gmd{\text{graph manifold}}
\def\PSL2R{\text{PSL}(2,\mathbb{R})}
\def\TSL2R{\widetilde{\text{SL}}(2,\mathbb{R})}
\def\MG{\widetilde{\text{SL}}(2,\mathbb{R}) \times_{\mathbb{Z}} \mathbb{R}}

\newenvironment{ack}{{\noindent\textbf{Acknowledgement.}}\ }

\newtheorem{theorem}{Theorem}[section]
\newtheorem{lemma}[theorem]{Lemma}
\newtheorem{definition}[theorem]{Definition}
\newtheorem{remark}[theorem]{Remark}
\newtheorem{proposition}[theorem]{Proposition}
\newtheorem{corollary}[theorem]{Corollary}

\title{Virtual representations for closed graph manifolds and \\ Seifert geometry} 
\author{YAO FAN}
\date{}

\begin{document}

    \maketitle

    \begin{spacing}{1.5}

        \begin{abstract}

            In this paper, we mainly discuss the $\rep$s of closed $\gmd$s to the Seifert motion group. 
            Then we prove that there exist $\gmd$s virtually having no faithful $\rep$s to the Seifert motion group. 
               
        \end{abstract}

        \section{Introduction}
            In Kirby's list of problems in low-dimensional topology \cite{Kir97}, 
            Thurston proposed a question if every finitely generated 3-manifold group $G$ has a faithful representation in $\text{GL}(4,\R)$. 
            However, Button gave a counterexample \cite{But14} which is a graph manifold made by gluing two Seifert manifolds. 
            The fundamental group of the graph manifold cannot be embedded in $\text{GL}(n,K)$ for $n\leq 4$ and for any field $K$. 
            
            Then we can ask if every finitely generated 3-manifold group $G$ has a faithful linear representation. 
            According to the book \cite{AscFW15} written by Aschenbrenner, Friedl and Wilton, which gives a detailed survey of 3-manifold groups,  
            the isometry group of the following geometric 3-manifold are subgroups of $\text{GL}(4,\R)$: 
            spherical geometry, $S^2\times \R$, Euclidean geometry, Nil, Sol and hyperbolic geometry . 
            Moreover, the fundamental group of an $\mathbb{H}^2\times \R$-manifold is a subgroup of $\text{GL}(5,\R)$ \cite[p.50]{AscFW15}. 
            Hence, these seven geometric manifolds have faithful linear $\rep$s. 
            However, the isometry group of $\TSL2R$ is not a linear group. 
            With the fact that the fundamental group of an $S^1$ bundle over a surface is linear over $\Z$ \cite[p.43]{AscFW15}, 
            Seifert geometry, whose model space is $\TSL2R$, virtually has a faithful linear $\rep$. 
            However, we cannot fix the size of the linear Lie group. 

            As for $\gmd$s, which consists of Seifert manifolds, the bounded case has better properties. 
            Result of Haglund and Wise \cite{HagW08}\cite{HagW12} shows that 
            if a 3-manifold admits a nonpositive curvature, then its fundamental group can be embedded in a linear Lie group. 
            Kapovich and Leeb proved that there exists a nonpositive curvature on a $\gmd$ with torus boundaries \cite{KapL96}. 
            Hence, a $\gmd$ with torus boundaries has a faithful linear $\rep$.  
            However, Kapovich and Leeb also gave an example of a closed $\gmd$ admitting no nonpositive curvature. 
            So, for closed $\gmd$s, we do not know if the fundamental groups have faithful linear $\rep$s. 

            Brooks and Goldman defined the Seifert volume by using a $\rep$ to the isometry group of $\TSL2R$ for 3-manifolds \cite{BroG84}. 
            Therefore, we consider the $\rep$s of closed $\gmd$ to the Seifert motion group, 
            which is the identity component of the isometry group of $\TSL2R$. 
            We prove a theorem as the following:
            \begin{theorem}\label{NFR}
                There exists a closed $\gmd$ virtually having no faithful $\rep$s to Seifert motion group. 
                In fact, there exists a closed graph manifold $M$ such that for any finite cover $M'$ of $M$ and 
                any representation $\rho: \pi_1(M') \to \MG$, $\rho$ is not faithful restricted to any Seifert block of $M'$. 
            \end{theorem}
            Here $\MG$ denotes the Seifert motion group.
            In addition, the theorem \ref{NFR} holds for closed $\gmd$s satisfying the condition of strictly diagonally dominance (Definition \ref{DSB}).  
            This condition is inspired from Derbez, Liu and Wang's paper \cite{DerLW20}. 
            If the $\gmd$ has nonempty torus boundaries, it can have a $\rep$ restricted on one submanifold is faithful. 
            We give an example of the $\gmd$ and construct the $\rep$ in Section \ref{EB}. 

            To get the theorem \ref{NFR} we prove a stronger version that there exist no virtually vertex faithful $\rep$s (Definition \ref{Rep}). 
            We suppose there is virtually a vertex faithful $\rep$ on a closed $\gmd$ satisfying the condition of strictly diagonally dominance. 
            Then the image of a submanifold under the $\rep$ is an Abelian group. 
            We list equations about the images of fibers which are deduced form the topological structure of the $\gmd$. 
            The solution of the equations contradicts with the assumption of the vertex faithful $\rep$. 
            So, we deduce the result that this $\gmd$ virtually has no faithful $\rep$s. 

            The paper is organized as the following: 

            In section \ref{Pre} we recall the relative definition of $\gmd$s and introduce two important topological invariants. 
            In section \ref{MG} we introduce the Seifert motion group and classify its elements in order to discuss $\rep$s. 
            Then in section \ref{FinC} we introduce JSJ characteristic finite covers 
            and give a property about topological invariants under this special finite cover. 
            In section \ref{Rep} we mainly discuss the representations to the Seifert motion group.  
            In section \ref{Proof} we prove the Theorem \ref{NFR} about the nonexistence of $\rep$ in the closed condition. 
            In addition, we construct a vertex faithful $\rep$ when a $\gmd$ has torus boundaries in section \ref{EB}. 

        \ 
        
        \begin{ack}
            I would like to give my thanks to my supervisor Yi Liu for his guidance and conversations. 
        \end{ack}

        \section{Preliminary}\label{Pre}
    
            Let $M$ be an orientable compact irreducible 3-manifold. 
            There is a collection of disjoint incompressible tori splits $M$ into components which are either atoroidal or Seifert fibered. 
            Moreover, the minimal such collection is unique up to isotopy. 
            This decomposition of 3-manifold is called Jaco-Shalen-Johannson torus decomposition (JSJ decomposition) \cite{JacS79}\cite{Joh79}. 

            In the following we review the definitions, related properties and some topological invariants of Seifert manifolds and $\gmd$s. 
            Definitions and notations mostly follows the paper of Buyalo and Svetlov \cite{BuyS04}. 

            \subsection{Seifert manifolds}

                A Seifert manifold is a 3-manifold foliated by circles. 
                A useful definition according to Scott \cite{Sco83} is to define the Seifert manifold as a ``circle bundle'' over an orbifold. 
                Let $M$ be a Seifert manifold and $O$ be its base orbifold. There is a projection $p: M \to O $ and an exact sequence of fundamental groups 
                \[ \{1\} \to K \to \pi_1(M) \to \pi_1(O) \to \{1\} , \]
                where $K$ is a cyclic subgroup of $\pi_1(M)$ generated by a regular fiber. 
                The group $K$ is infinite cyclic except the case where $M$ is covered by $S^3$. 

                Let the orbifold $O$ has an underlying surface $F$ and $m$ singular points $s_i$ with degree $q_i$, $1\leq i\leq m$. 
                Then the Euler characteristic number of an orbifold is defined to be 
                \[ \chi(O) = \chi(F) - \sum_{i=1}^{m} \left( 1 - \frac{1}{q_i}\right). \]
                If $\chi(O) < 0 $, then there is a finite cover of $M$ having a hyperbolic structure on its base surface.  
                The fundamental group of the base surface can be embedded into $\PSL2R$. 
                In the following we only consider the Seifert manifold having a base orbifold with a negative Euler characteristic number. 
                Hence, there is a finite cover of the Seifert manifold is a trivial circle bundle over a surface. 
                In the rest of this paper, the Seifert manifold is a trivial circle bundle without special mentions. 

                \ 

                \noindent\textbf{Waldhausen basis. }
                Let $M$ be an oriented Seifert manifold with nonempty boundaries.  
                Let $W$ be the set of boundary components. 
                Consider an oriented fiber $f$ of $M$ and 
                there is a homological class $f_w \in H_1(T_w;\Q)$ corresponding to fiber $f$ in each boundary torus $T_w,w\in W$. 
                Then there is an induced orientation from $M$ to $T_w$. 
                Hence, the intersection form $ (\cdot,\cdot): H_1(T_w;\Q) \times H_1(T_w;\Q) \to \Q$ can be defined. 
                For each $f_w$ there is a $z_w\in H_1(T_w;\Q)$ satisfying $(z_w,f_w)= 1$. 
                Here we modify the notation of the intersection form which is different from the notation of Buyalo and Svetlov. 

                \begin{definition}
                    For each torus boundary $T_w$ of a Seifert manifold $M$ there is a $z_w$ such that $(z_w, f_w) = 1$. 
                    Then a collection of elements $\{z_w,f_w\}_{w\in W}$ is called a Waldhausen basis of the $M_v$
                    if they satisfy the condition that $z =\sum_{w\in W} z_w $ lies in the kernel of the inclusion $i_*: H_1(\partial M;\Q) \to H_1(M;\Q)$.
                \end{definition}

                The Waldhausen basis is not unique. It is defined up to a transformation $z_w \mapsto z_w + n_w f_w$ where $\sum_{w\in W} n_w = 0$. 
                If $M$ is a trivial circle bundle, then we can choose $z_w$ from $H_1(T_w;\Z)$. 
                This choice of Waldhausen basis is equivalent to the choice of a trivialization of $M = F \times S^1$. 

                \ 

                \noindent\textbf{Framed Seifert manifolds. }
                Let $M$ be a Seifert manifold and each boundary $T_w$ has a fixed element $c_w \in H_1(T_w;\Z)$ with $(c_w,f_w)\neq 0$. 
                Then $M$ is called framed and the collection $\{ c_w\}_{w\in W}$ is a framing of $M$. 

                For a framed Seifert manifold, there is a formula about its fiber and the framing. 
                We state the formula as a lemma. 

                \begin{lemma}\label{FR}
                    Let $\{ c_w \}_{w\in W}$ be a framing of a Seifert manifold $M$. Then 
                    \begin{equation}\label{REC}
                        \sum_{w\in W} \frac{1}{(c_w, f_w)} \cdot (i_w)_* c_w = k \cdot f,  
                    \end{equation}
                    where $(i_w)_*: H_1(T_w;\Z) \to H_1(M;\Z)$ is the inclusion homomorphism induced by $i_w: T_w \to M_v$.  
                    The homological class $f \in H_1(M;\Z)$ represents the regular fiber, and $(i_w)_*f_w = f$. 
                    The number $k$ on the right hand is a rational number. 
                \end{lemma}

                \begin{proof}
                    We choose a Waldhausen basis $\{z_w,f_w\}_{w\in W}$ for $M$. 
                    $c_w$ can be represented as $c_w = a_w f_w + b_w z_w$. 
                    Since $b_w = (c_w,f_w) \neq 0$, 
                    \[
                        \sum_{w\in W} \frac{1}{(c_w,f_w)} \cdot (i_w)_* c_w = 
                        \sum_{w\in W} \frac{1}{b_w} \cdot (i_w)_* (a_w f_w + b_w z_w) = 
                        \left( \sum_{w\in W} \frac{a_w}{b_w}  \right)f.
                    \]
                    Here the coefficient $k = \sum_{w\in W} \frac{a_w}{b_w}$ on the right hand is independent of the Waldhausen basis 
                    since the left hand is independent of the choice. 
                \end{proof}
            
            \subsection{Graph manifolds}

                A graph manifold $M$ is an irreducible 3-manifold with all components Seifert manifolds under the JSJ decomposition. 
                The $\gmd$ $M$ is associated with a directed graph $\Gamma(V,E) $ dual to the JSJ decomposition. 
                We call each component a Seifert block.  
                The vertex set $V$ of $\Gamma$ is the set of Seifert blocks and the set $E$ of oriented edges of $\Gamma$ is identified to boundaries of all blocks. 
                If $e\in E$ is an edge directed form $v_1$ to $v_2$, then $T_{v_1,v_2} \subset \partial M_{v_1} $ and $T_{v_2,v_1} \subset \partial M_{v_2} $. 
                Hence, $e$ is identified to an ordered pair $(v_1,v_2)$. 

                If a graph $\Gamma$ is connected, for two points $v_1,v_2$, 
                Let $\P$ denotes the set of all paths $L$ connecting $v_1,v_2$. 
                We define the length $\# L$ to be the number of edges in $L$.  
                The metric on $\Gamma$ is defined as 
                \[ d(v_1,v_2) = \min_{L\in \P} \# L. \]
                Then we can define the distance between two Seifert blocks by their corresponding vertices as 
                \[ d(M_{v_1},M_{v_2}) = d(v_1,v_2). \]

                We fix an orientation of a graph manifold $M$. 
                Then a torus boundary of a Seifert block has an induced orientation. 
                Let $M_{v_1},M_{v_2}$ be adjacent Seifert blocks. 
                We denote the torus by $T_{v_1,v_2}$ with the induced orientation from $M_{v_1}$ dual to the directed edge $e=(v_1,v_2)$. 
                Then the orientation of $T_{v_2,v_1}$ is opposite to the orientation of $T_{v_1,v_2}$ 
                
                A gluing map is a homeomorphism associated with a directed edge $e=(v_1,v_2)$ maps $T_{v_1,v_2}$ to $T_{v_2,v_1}$. 
                If we choose $\{ f_{v_1}, z_{v_1} \}$ as a basis of $\pi_1(T_{v_1,v_2})$ and $\{ f_{v_2}, z_{v_2} \}$ as a basis of $\pi_1(T_{v_2,v_1})$, 
                then the gluing map induces an isomorphism on the fundamental groups. It can be represented as a 2$\times$2 matrix 
                \[
                    g_{v_1,v_2} = 
                    \begin{pmatrix}
                        a & b \\
                        c & d
                    \end{pmatrix}, 
                    \quad ad-bc=-1. 
                \]
                So we call the isomorphism $g_{v_1,v_2}$ a gluing matrix.

                In the following we introduce two topology invariants of graph manifolds referred to Buyalo and Svetlov \cite{BuyS04}. 
                We make a few modifications by using the gluing matrix in definitions. 

                \ 
                
                \noindent\textbf{The intersection index. }
                Let $M_{v_1}$ and $M_{v_2}$ be two adjacent Seifert blocks in an oriented $\gmd$. 
                We use $f_{v_1,v_2}$ to denote the homological class of the regular fiber of $M_{v_1}$ in $H_1(T_{v_1,v_2};\Z)$ 
                and $f_{v_2,v_1}$ to denote the fiber of $M_{v_2}$ in $H_1(T_{v_2,v_1};\Z)$. 
                By the gluing matrix $g_{v_1,v_2}$, the fiber $f_{v_1,v_2}$ maps to an element in $H_1(T_{v_2,v_1};\Z)$.  
                
                Define 
                \[ b_{v_2,v_1} = (f_{v_2,v_1}, g_{v_1,v_2}(f_{v_1,v_2})).  \]
                By definition of a $\gmd$ and its induced orientation, the integer number $b_{v_2,v_1}$ satisfies $b_{v_2,v_1} = b_{v_1,v_2} \neq 0 $. 
                This integer $b_{v_2,v_1}$ is called the intersection index of fibers $f_{v_2,v_1}$ and $f_{v_1,v_2}$. 
                It changes the sign when the orientation of one fiber $f_{v_2,v_1}$ or $f_{v_1,v_2}$ changes.  

                \ 

                \noindent\textbf{The charge. }
                For a Seifert block $M_v$, it has a natural framing $\{c_w\}_{w\in W}$ by setting $c_w = g_{w,v}(f_{w,v})$. 
                Then the coefficient of $f$ on the right hand in the equation \ref{REC} is called the charge of $M_v$. 
                Under a chosen Waldhausen basis $\{ f_{v,w},z_{v,w} \}_{w\in W}$ of $M_v$, 
                $c_w = g_{w,v}(f_{w,v}) = a_{w,v} f_{v,w} + b_{w,v} z_{v,w}$. 
                Then the charge can be written as
                \[ k_v = \sum_{w\in W} \frac{a_{w,v}}{b_{w,v}}. \]
                According to Lemma \ref{FR}, the charge is independent of the choice of a Waldhausen basis. 
                But $ \frac{a_{w,v}}{b_{w,v}} $ which is called the slope for the homological class $f_{w,v}$ is related to the 
                Waldhausen basis $\{ f_{v,w},z_{v,w} \}_{w\in W}$. 

                If a Seifert block has a boundary $T$ which is also the boundary of the whole graph manifold, 
                then this boundary $T$ has no contribution to the charge of the Seifert block. 
                We define its slope always be zero in the calculation of the charge.  
                
                \ 

                \noindent\textbf{Strictly diagonally dominant graph manifold. }

                In the paper of Derbez, Liu and Wang \cite{DerLW20}, they introduce the terminology, strictly diagonally dominance, 
                to describe a property of graph manifolds. This terminology is original from the theory of matrices. 

                If a $n\times n$ matrix $A=(a_{ij})$ satisfies $|a_{ii}| > \sum_{j\neq i} |a_{ij}|$ for all $i$, 
                then the matrix $A$ is said to be strictly diagonally dominant.  
                
                \begin{proposition}
                    If $A$ is a strictly diagonally dominant matrix, then $A$ is invertible. 
                \end{proposition}

                \begin{proof} 
                    Suppose $A$ is not invertible, then there is a nonzero vector $\bm{x}$ such that $A\bm{x}=\bm{0}$. 
                    Let $x_i$ be the element with maximum absolute value in $\bm{x}$. 
                    Then there is an equation $a_1 x_1 + \cdots +a_i x_i + \cdots + a_n x_n = 0$. 
                    However, 
                    \[ \left\vert \sum_{j\neq i} a_{ij} x_j \right\vert \leq \sum_{j\neq i} |a_{ij} x_j| 
                        \leq \left( \sum_{j\neq i} |a_{ij}| \right) |x_i| < |a_{ii} x_i|. 
                    \]
                    This deduces a contradiction. Hence, $A$ is invertible. 
                \end{proof}

                In the aspect of the graph manifolds, the definition is as the following. 

                \begin{definition}\label{DSB}
                For a Seifert block $M_v$, if the charge $k_v$ and all intersection numbers $b_{v,w}$ satisfying 
                \begin{equation}
                    |k_v| > \sum_{w\in W} \frac{1}{|b_{v,w}|}, 
                \end{equation}
                in which $W$ is the set of all adjacent blocks of $M_v$. Then this Seifert block is called strictly diagonally dominant. 
                If a $\gmd$ has every Seifert block strictly diagonally dominant, then it is called a strictly diagonally dominant $\gmd$. 
                \end{definition}

                Strictly diagonally dominance is a relation about the charges and intersection indices of a $\gmd$. 
                In the following section \ref{FinC} we will prove this relation still holds under a special finite cover. 

        \section{Seifert Motion Group}\label{MG}
            
            Seifert geometry, or $\TSL2R$ geometry, is one of the eight 3-dimensional geometries as classified by Thurston \cite{Thu97}. 
            We take Lie group $\TSL2R$ as the model of the Seifert geometry's space. 
            The identity component of its isometry group can be identified with $\MG$, which is called as Seifert motion group. 

            Eisenbud, Hirsch and Neumann gave criteria for a Seifert manifold to admit a transverse foliation \cite{EisHN81}. 
            In their paper they studied the homomorphism $\pi_1(M) \to \TSL2R$.  
            The name of Seifert geometry is from Brooks and Goldman.  
            They defined the Seifert volume by using a $\rep$ to the Seifert motion group for 3-manifolds \cite{BroG84}. 
            Then Derbez, Liu and Wang gave a formula to compute the $\rep$ volume for a $\gmd$ with a $\rep$ to the Seifert motion group \cite{DerLW20}. 
            So, in this section we state the construction of the Seifert motion group and classify elements in this group.   

            \subsection{Construction of Seifert motion group}

                In this part we use the notation from Derbez, Liu and Wang \cite{DerLW20}. 
                Here $\TSL2R$ is the universal cover of $\PSL2R$. There is a map $k: \R \to \PSL2R$,
                \[ k(r) = \begin{pmatrix}
                    \cos \pi r & -\sin \pi r \\
                    \sin \pi r & \cos \pi r
                \end{pmatrix}, r\in \R. \]
                Then there exists a lift $\tilde{k}: \R \to \TSL2R$. 
                The center $Z(\TSL2R)$ of $\TSL2R$ is the group $\{ \tilde{k}(n),n\in \Z \}$, 
                which projects to the identity in $\PSL2R$. 

                The Seifert motion group is the quotient of a product group $\TSL2R \times \R$ by a subgroup $\{ (\tilde{k}(n),-n),$ $n\in \Z \}$. 
                Then the Seifert motion group is denoted by $\MG$. 
                Its elements can be denoted as $g[s]$ for $g \in \TSL2R $ and $s \in \R$. 
                The multiplication rule is 
                \[ g[s]g'[s'] = gg'[s+s'], \]
                and there is a relation 
                \[ g[s+n] = g\tilde{k}(n)[s], n \in \Z. \]

            \subsection{Classifying elements in Seifert motion group}

                Elements in SL(2,$\R$) can be classified by their traces. 
                Take a matrix $A \in$ SL(2,$\R$). 
                If the absolute value of the trace $|\text{Tr}(A)| < 2 $, the matrix $A$ is said to be elliptic. 
                If $|\text{Tr}(A)| > 2 $, it is said to be hyperbolic. 
                If $|\text{Tr}(A)| = 2 $ and $A$ is not the identity matrix or minus identity matrix, it is said to be parabolic. 
                Since $\PSL2R = \text{SL}(2,\R)/\pm I_2 $, its elements can be classified by their representative elements. 
                This way can be checked that is well-defined. 
            
                Then we consider a natural projection $pr: \MG \to \PSL2R$. The elements in the Seifert motion group can be classified by their images. 
                Let $g\in \MG$, 
                \begin{itemize}
                    \item if $ pr(g)$ is the identity in $\PSL2R$, we call it central type;
                    \item if $ pr(g)$ is an elliptic element in $\PSL2R$, we call it elliptic type;
                    \item if $ pr(g)$ is a hyperbolic element in $\PSL2R$, we call it hyperbolic type;
                    \item if $ pr(g)$ is a parabolic element in $\PSL2R$, we call it parabolic type. 
                \end{itemize}

                Then there is a proposition about the subgroup $C(g)=\{ h\in \MG | gh=hg \}$. 
                \begin{proposition}\label{CP}
                    Let $a,b$ be two elements in the Seifert motion group $\MG$. 
                    Then $ab=ba$ if and only if $pr(a)pr(b) = pr(b)pr(a)$ in $\PSL2R$. 
                    Hence, for a noncentral element $g\in \MG$, the subgroup $C(g)$ is an Abelian group.  
                \end{proposition}
                
                \begin{proof}
                    It is direct that $ab=ba$ deduces $pr(a)pr(b) = pr(b)pr(a)$ in $\PSL2R$. 
                    Conversely, if $pr(a)pr(b) = pr(b)pr(a)$, 
                    by the construction of the Seifert motion group, $aba^{-1}b^{-1}$ lies in the center of $\MG$. 
                    
                    For $a,b \in \MG$, we have continuous paths starting from the identity element to $a$,$b$ respectively. 
                    Notice  
                    \begin{align*}
                        (\MG)\times& (\MG)  \to \MG,  \\
                        (a,b) &\mapsto aba^{-1}b^{-1} 
                    \end{align*}
                    is a continuous map.  
                    However, the group $\{ \tilde{k}(n),n\in \Z \} \times_{\Z} \R $, the center of the Seifert motion group, is discrete in the first component. 
                    So, the first component of $aba^{-1}b^{-1}$ is the identity element in $\TSL2R$. 
                    Notice that the second component of $aba^{-1}b^{-1}$ is always zero. 
                    Hence, $aba^{-1}b^{-1}$ equal to the identity element in $\MG$. 

                    If two nontrivial elements in $\PSL2R$ are commutable, they can be diagonalized at the same time and lie in the same type. 
                    Hence, for a noncentral element $g\in \MG$, $C(pr(g)) = \{ pr(h)\in \PSL2R | pr(g)pr(h)=pr(h)pr(g) \}$ is an Abelian group. 
                    So, $C(g)$ is an Abelian group. 
                \end{proof}

                \ 
                
                For elements in the Seifert motion group we have another classification with respect to the order of their projective images.  
                \begin{definition}\label{PO}
                    If an element in Seifert motion group is projected to an element with finite order in $\PSL2R$, 
                    then we call it an element with projectively finite order. 

                    If an element in Seifert motion group is projected to an element with infinite order in $\PSL2R$, 
                    then we call it an element with projectively infinite order. 
                \end{definition}
                
                \begin{remark}\label{ROE}
                    Let $f = \tilde{k}(n)$ be a central element in $\MG$. 
                    Consider an equation $x^m = f, m\in \Z $. 
                    If $ n = lm + r, l,r\in \Z$, then $\tilde{k}(l)[\frac{r}{m}]$ is a solution of the equation. 
                    However, $\tilde{k}(l)[\frac{r}{m}]$ is not a unique solution. 
                    We project the equation to $\PSL2R$ then $(pr(x))^m = pr(f) = I_2$. 
                    So, any solution of the equation $x^m = f$ is an element with projectively finite order. 
                \end{remark}

        \section{JSJ characteristic Finite Cover}\label{FinC}
        
            In this section we discuss the finite covers of graph manifolds and the relation of two topological invariants.   
            If some finite cover of $M$ has a property $P$, then $M$ is said to have a property $P$ virtually. 
            Sometimes we only consider the existence of $\rep$s. 
            Therefore, we often pass to a finite cover to prove theorems or construct representations.  
            
            According to Kapovich and Leeb's result \cite[Lemma 2.1]{KapL98}, a $\gmd$ has a good structure passing to a finite cover. 
            We state the lemma as the following: 
            \begin{lemma}[Kapovich, Leeb]\label{KLL}
                Any $\gmd$ has an orientable finite cover where all Seifert blocks are trivial circle bundle over a surface with genus greater than 1. 
                Furthermore, we can arrange the intersection index of the fibers of adjacent Seifert blocks are $\pm 1$. 
            \end{lemma}  
                    
            By the above lemma, passing to a finte cover each Seifert block $M_v$ can be homeomorphic to $F_v \times S^1 $. 
            In the following we assume the $\gmd$ satisfies the conditions in lemma \ref{KLL}.  
            Moreover, we consider the dual graph $\Gamma$ is a nontrivial simple graph, which is a graph contains no multiple edges and self-loops. 
            Hence, every JSJ torus separates two distinct Seifert blocks from each other. 
            This condition can also be satisfied by passing to a finite cover. 
        
            \begin{definition}
                Set $p:M' \to M$ be a regular finite cover of $M$. 
                Let the set $\mathcal{T}$ be a union of JSJ tori of $M$, the set $\mathcal{T}'$ be the preimage of $\mathcal{T}$ 
                and $n$ be a positive integer. 
                If for each torus $T$ of $\mathcal{T}$ and each component $T'$ of $\mathcal{T}'$ over $T$, 
                the restriction $p|:T' \to T$ is a characteristic cover with index $n\times n$, 
                then this finite cover $p:M' \to M$ is called a JSJ $n$-characteristic finite cover. 
            \end{definition}
            A regular finite cover $p_1:M' \to M$ is not necessarily a JSJ characteristic finite cover. 
            However, we can always find a JSJ characteristic finite cover $p_2:M'' \to M$ such that $p_2$ factors through $p_1$. 
            Hence, we can prove some properties on JSJ characteristic finite cover and then deduce to a general case.   
            For a JSJ characteristic finite cover, there is a property about topological invariants. 
            \begin{proposition}\label{SDD}
                Let $M$ be a strictly diagonally dominant $\gmd$,  
                then any JSJ characteristic finite cover of $M$ is also strictly diagonally dominant. 
            \end{proposition}
        
            \begin{proof}
                Suppose $p: M' \to M $ is a JSJ characteristic finite cover. 
                Take a Seifert block $M'_{v'}$ which is a component of a Seifert block $M_v$'s preimage. 
                Since this finite cover is JSJ characteristic, the slope of each torus boundary is invariant. 
                Then the covering degree of the fiber is fixed for each Seifert block. 
                Hence, the gluing matrices are also invariant under the covering. 
                    
                Let $W$ be the set of all adjacent block of $M_v$. 
                If an edge $e$ with $\partial e = (v,w),w \in W$ has a preimage $p^{-1}(e)$ connecting $v'$, 
                we set $\partial (p^{-1}(e)) = \{ (v',w') | w'\in W'_{v'} \}$ 
                in which $W'_{v'}$ is the set of all Seifert blocks connecting $v'$ by one edge of $p^{-1}(e)$. 
                For each torus $T'_{v',w'}$ dual to $(v',w')$, the intersection number $b_{v',w'} = b_{v,w}$ 
                since the gluing maps have the same matrix $\rep$ for $T'_{v',w'}$ and $T_{v,w}$. 
                Hence, the sum of slope of $T'_{v',w'},w' \in W'_{v'} $ is just $|W'_{v'}|$, the size of $W'_{v'}$, times the slope of $T_{v,w}$. 
                Notice the size of $W'_{v'}$ is the same for all edges connecting $v'$. 
                Since the degree of fiber is fixed, 
                the degree of JSJ characteristic cover for each torus is the same. 
                Hence, the components of preimage $p^{-1}(e)$ for each $e$ is also the same. 
            
                Because $M_v$ is strictly diagonally dominant, we have 
                    \[ |k_v| > \sum_{w\in W} \frac{1}{|b_{v,w}|}. \] 
                As for $M'_{v'}$, the charge 
                \begin{align*}
                    k_{v'} &= \sum_{w\in W } \sum_{w'\in W'_{v'}} \frac{a_{v',w'}}{b_{v',w'}}
                            = \sum_{w\in W } \left( |W'_{v'}| \times \frac{a_{v',w'}}{b_{v',w'}} \right) \\
                            &= |W'_{v'}| \times \sum_{w\in W } \frac{a_{v',w'}}{b_{v',w'}} 
                            = |W'_{v'}| \times \sum_{w\in W } \frac{a_{v,w}}{b_{v,w}} \\
                            &= |W'_{v'}| \times k_v.
                \end{align*}
            
                Then 
                \begin{align*}
                    |k_{v'}| &= |W'_{v'}| \times |k_v| > |W'_{v'}| \times \sum_{w\in W} \frac{1}{|b_{v,w}|} \\
                            &= \sum_{w\in W} \left(|W'_{v'}| \times \frac{1}{|b_{v,w}|} \right)
                            = \sum_{w\in W} \left(|W'_{v'}| \times \frac{1}{|b_{v',w'}|} \right) \\
                            &= \sum_{w\in W} \sum_{w'\in W'_{v'}} \frac{1}{|b_{v',w'}|}
                \end{align*}
                So the Seifert block $M'_{v'}$ is also strictly diagonally dominant. 
                
                Since the choice of vertex $v$ is arbitrary, each Seifert block of $M'$ is strictly diagonally dominant. 
                Hence, $M'$ is a strictly diagonally dominant $\gmd$.  
            \end{proof}
                
        \section{Representations to Seifert Motion Group}\label{Rep}
            
            If a $\rep$ of a $\gmd$ is faithful, then it restricted on each Seifert block is faithful. 
            In convenience, we define a ``local" faithful $\rep$ as the following. 
            \begin{definition}\label{VFR}
                If a $\rep$ restricted on a Seifert block dual to a vertex $v$ is faithful, then we call the $\rep$ a vertex faithful representation, 
                and more precisely, faithful at $v$. 
            \end{definition} 
            Sometimes we call a $\rep$ vertex faithful without indicating a vertex $v$. 

            In the following we discuss the $\rep$s restricted on each Seifert block and classify them by the images in the Seifert motion group. 
            Particularly, we list equations for Seifert blocks with Abelian images. 
            
                \subsection{Faithful representations}

                    Consider a Seifert block $M_v$ which is a trivial circle bundle over a compact surface with genus $g \geq 2$.   
                    Suppose there is a faithful representation on $M_v$ to the Seifert motion group. 
                    We have a conclusion as the following: 
                    \begin{proposition}\label{IFR}
                        Let a Seifert block $M_v$ be a trivial circle bundle over a compact surface $F_v$ with genus $g \geq 2$. 
                        If $M_v$ has a faithful representation $\rho : \pi_1(M_v) \to \MG $, 
                        then the image of the fiber of $M_v$ is a nontrivial element of central type 
                        and the images of boundaries of the base surface $F_v$ are all noncentral elements with projectively infinite order. 
                    \end{proposition}

                    \begin{proof}
                        Since $M_v$ is a trivial circle bundle over a compact surface $F_v$ with genus $g \geq 2$, 
                        the fundamental group of $M_v$ has a presentation as 
                        \[ 
                        \pi_1(M_v) = \Bigl< a_1,b_1,\cdots,a_g,b_g,c_1,\cdots,c_l \Big\vert \prod_{i=1}^{g} [a_i,b_i] = \prod_{j=1}^{l} c_j \Bigr> \times \langle f \rangle. 
                        \]
                        Here $[a_i,b_i]$ denotes $a_i b_i a_i^{-1} b_i^{-1}$. The integer $g$ is the genus of the base surface and $l$ is the number of boundary components. 

                        The image of the fiber of $M_v$ projects to identity under the projection $pr:\MG \to \PSL2R$, 
                        otherwise the representation $\rho$ is Abelian which contradicts the assumption that $\rho$ is faithful. 
                        Hence, the image of the fiber $\rho(f)$ is a nontrivial central element. 

                        As for the images of boundaries $\{ c_j \}$ of the base surface $F_v$, they cannot lie in the center. 
                        If one of them, $\rho(c_j)$, has projectively finite order, then there exists an integer $n$ such that $\rho(c_j)^n$ lies in the center. 
                        Since $\rho$ is faithful, $(c_j)^n$ commute with every generator in $\pi_1(M_v)$. 
                        This result contradicts the presentation of $\pi_1(M_v)$. 
                        Hence, the images of boundaries of base surface $F_v$ are all noncentral elements with projectively infinite order. 
                    \end{proof}

                    \begin{corollary}\label{ABF}
                        If a $\rep$ $\rho$ of a $\gmd$ $M$ restricted on a Seifert block $M_v$ is faithful, 
                        then the adjacent blocks of $M_v$ all have noncentral fibers with projectively infinite order. 
                    \end{corollary}

                    \begin{proof}
                        By proposition \ref{IFR}, the image of the fiber of $M_v$ is a nontrivial element of central type and 
                        the images of boundaries of the base surface $F_v$ are all noncentral elements with projectively infinite order. 

                        Since the intersection indices are nonzero integers, the gluing map induces a form of matrix $g_{v,w}=$
                        $\left(\begin{smallmatrix}
                            a & b \\
                            c & d
                        \end{smallmatrix}\right)$ with $b\neq 0$. 
                        The fiber $f_w$ of an adjacent blocks can be written as $f_w = a f_v + b z_{v,w}$ in the fundamental group $\pi_1(T_{v,w})$. 
                        Hence, the image $\rho(f_w)$ is a noncentral element with projectively infinite order.  
                    \end{proof}

                \subsection{Abelian representations}
                
                    Let $\rho$ be a $\rep$ of a Seifert block $M_v$ to $\MG$. 
                    If the image of the fundamental group $\pi_1(M_v)$ under $\rho$ is an Abelian group, 
                    the restricted $\rep$ $\rho: \pi_1(M_v) \to \MG$ is called an Abelian $\rep$ of $M_v$. 
                    Since the Abelian $\rep$ factors through the homology group $H_1(M_v;\Q)$, 
                    there is a Waldhausen basis to represent the boundaries of the base surface. 
    
                    Suppose $M_v$ has an Abelian $\rep$ and it has torus boundary $T_{v,w}, w\in W$, where $W$ is the set of all adjacent Seifert blocks of $M_v$. 
                    We choose a Waldhausen basis $\{f_{v,w},z_{v,w}\}_{w \in W}$. 
                    Under gluing matrix $g_{w,v}$, the homological class $f_{w,v}\in H_1(T_{w,v};\Z)$ has an image $g_{w,v}(f_{w,v})$ in $H_1(T_{v,w};\Z)$. 
                    Notice $H_1(T_{v,w};\Z)$ has a basis $f_{v,w}, z_{v,w}$. 
                    Then $f_{w,v}$ can be represented by a linear combination $g_{w,v}(f_{w,v}) = a_{w,v} f_{v,w} + b_{w,v} z_{v,w}$, $a_{w,v},b_{w,v} \in \Z$.  
                    
                    We define a nonzero integer
                    \[b_v = \prod_{w\in W} b_{w,v}. \] 
                    Then we have 
                    \begin{equation}\label{BF}
                        \frac{b_v}{b_{w,v}} g_{w,v}(f_{w,v}) = \frac{b_v a_{w,v}}{b_{w,v}} f_{v,w} + b_v z_{v,w} . 
                    \end{equation}
                    Since $\{z_{v,w}\}$ is in the Waldhausen basis and $M_v$ has the charge $k_v$ which satisfies the equation 
                    \[ k_v = \sum_{w\in W} \frac{a_{w,v}}{b_{w,v}}. \]
                    Sum up on both sides for equation (\ref{BF}), then 
                    \begin{equation}\label{BSF}
                        \sum_{w\in W} \frac{b_v}{b_{w,v}} f_w = b_v k_v f_v. 
                    \end{equation}
                    In convenience, we ignore gluing matrices and inclusion maps. In addition, we denote $f_{w,v}$ by $f_w$ as well as $f_{v,w}$ by $f_v$. 
    
                    Here we introduce the number $b_v$ in order to avoid discussing any $n$-th root of a central element in $\MG$. 
                    In the above equations we can find $b_v/b_{w,v}$ is an integer, so it involves only group multiplication.

                \subsection{Abelian components}

                    In this part, we firstly introduce some definitions in graph theory from Harary \cite{Har69}. 

                    \begin{definition}
                        Let $\Gamma(V,E)$ be a graph, then $\Gamma_0(V_0,E_0)$ is called a subgraph of $\Gamma$ if $V_0$ and $E_0$ are subsets of $V$ and $E$ respectively. 
                    \end{definition}

                    \begin{definition}
                        Let $\Gamma(V,E)$ be a graph. A vertex-induced subgraph $\Gamma_0(V_0,E_0)$, simply called ``an induced subgraph" of $\Gamma$,  
                        induced by the vertex set $V_0$ is a graph with vertex set $V_0$ and edge set $E_0$ consisting of those edges both of whose endpoints are in $V_0$. 
                    \end{definition}
    
                    Let $\rho$ be a $\rep$ of a graph manifold $M$ and let $\Gamma$ be a dual graph of $M$. 
                    Take $\Gamma_0$ be a connected subgraph of $\Gamma$. 
                    Then we use $M_0$ to denote a submanifold consisting of Seifert blocks corresponding to $\Gamma_0$. 
                    If the image of the fundamental group $\pi_1(M_0)$ under the representation $\rho$ is an Abelian group, 
                    then we call the connected subgraph $\Gamma_0$ as an Abelian subgraph and call $M_0$ as an Abelian component. 
                    In addition, if the fibers of Seifert blocks in $M_0$ are all noncentral type under the representation $\rho$, 
                    then we call $M_0$ the noncentral Abelian component and $\Gamma_0$ noncentral Abelian subgraph. 
                 
                    \begin{definition}
                        Let $\rho$ be a $\rep$ of a graph manifold $M$. 
                        The set of all noncentral Abelian subgraphs has a partial order which is the inclusion of subgraphs. 
                        Then we define a maximal noncentral Abelian subgraph which is not included by any Abelian subgraph except itself. 
                        Similarly, we can define a maximal noncentral Abelian component.      
                    \end{definition}

                    The set of noncentral Abelian subgraphs may be an empty set under some representation. 
                    We suppose the set of noncentral Abelian subgraphs is not empty under a representation. 
                    Then there exists at least one maximal noncentral Abelian subgraphs. 
                    There are two propositions about the maximal one. 
                    
                    \begin{proposition}\label{ABC}
                        Let $M_0$ be a maximal noncentral Abelian component under a representation $\rho$. 
                        Then the adjacent Seifert blocks of $M_0$ all have central fibers. 
                    \end{proposition}

                    \begin{proof}
                        Suppose $M_w$ is an adjacent Seifert block of $M_0$ with a noncentral fiber $f_w$. 
                        Hence, the fundamental group $\pi_1(M_w)$ is an Abelian group. 
                        Since $\rho(f_w)$ commutes with the images of all fibers of Seifert blocks in $M_0$, 
                        all elements of $\pi_1(M_w)$ commute with all elements of $\pi_1(M_0)$. 
                        The fundamental group of a submanifold consisting of $M_0$ and $M_w$ is an Abelian group.  
                        This result contradicts with the definition of maximal noncentral Abelian component. 
                        Hence, the fiber of $M_w$ is of central type. 
                    \end{proof}

                    \begin{proposition}\label{ISG}
                        Let $M$ be a closed graph manifold with dual graph $\Gamma (V,E)$. 
                        If $\Gamma_0(V_0,E_0)$ is a maximal noncentral Abelian subgraph under a representation $\rho$, 
                        then there is a characteristic finite cover $p: M' \to M$ such that 
                        the subgraph induced by vertex set $V'_0$ is a maximal noncentral Abelian subgraph under the representation $p\circ \rho$, 
                        where $V'_0$ is the preimage of $V_0$. 
                    \end{proposition}

                    \begin{proof}

                        If $\Gamma_0$ is an induced subgraph, then it is a trivial case and the conclusion holds. 
                        Suppose $\Gamma_0(V_0,E_0)$ is not an induced subgraph, 
                        and it has $m$ fewer edges than a vertex-induced subgraph $\hat{\Gamma}_0(V_0,\hat{E}_0)$ induced by $V_0$. 
                        We denote $m$ edges by $e_i$ and the corresponding torus by $T_i$. 
                        For each $T_i$, there is a gluing matrix $g_i$.   
                        Let $M_0$ denote the corresponding submanifold of $\Gamma_0$ and $\hat{M}_0$ denote the corresponding submanifold of $\hat{\Gamma}_0$. 
                        If $\pi_1(M_0)$ has a presentation $\langle S|R \rangle$. 
                        Then $\pi_1(\hat{M}_0)$ is an HNN extension of $\pi_1(M_0)$. 
                        It has a presentation 
                        \begin{align*}
                             \langle S,t_i,1\leq i \leq m \ | \ R,\ t_if_it_i^{-1} = g_i(f_i), \ t_iz_it_i^{-1} = g_i(z_i),1\leq i \leq m  \rangle, 
                        \end{align*}
                        where $f_i,z_i$ are the generators of $\pi_1(T_i)$. 

                        Consider the projection $pr:\MG \to \PSL2R$. 
                        Let $\alpha$ be an element of $\pi_1(M_0)$. 
                        Notice 
                        \[ |\text{Tr}(pr(\rho(t_i\alpha t_i^{-1})))|=|\text{Tr}(pr(\rho(\alpha)))|. \]
                        Hence, we have two possible cases, 
                        \[ pr(\rho(t_i\alpha t_i^{-1})) = pr(\rho(\alpha)) \tag{$\star$} \] 
                        or
                        \[ pr(\rho(t_i\alpha t_i^{-1})) = pr(\rho(\alpha))^{-1}. \tag{$\star \star$} \]
                         
                        Since $\rho(\pi_1(M_0))$ is an Abelian group, by proposition \ref{CP}, 
                        any element commuting with $\rho(\alpha)$ will commute with all elements in $\rho(\pi_1(M_0))$. 
                        Hence, for each $t_i$, it either satisfies the first relation $(\star)$ for all elements in $\pi_1(M_0)$ 
                        or satisfies the second relation $(\star \star)$ for all elements in $\pi_1(M_0)$. 

                        Suppose there are $n$ edges $\{ e_j,1\leq j\leq n \}$ corresponding to $\{ t_j,1\leq j\leq n\}$ 
                        where every $t_j$ satisfies the second relation $(\star \star)$. 
                        Notice each $\rho(t^2_j)$ commutes with every element in $\rho(\pi_1(M_0))$. 
                        Then we can construct a double cover satisfying the conclusion. 

                        We state a construction of the cover on the level of graph, hence it is a JSJ characteristic finite cover. 
                        Cut $e_j$ in $\Gamma$ to get two ``half edges'' $e^+_j,e^-_j$ for each $j$ and take a copy of the broken graph. 
                        Then we glue $e^+_j$ to the half edge $\bar{e}^-_j$ and glue $e^-_j$ to the half edge $\bar{e}^+_j$ for each $j$, 
                        where $\bar{e}^+_j,\bar{e}^-_j$ are on the copy of the broken graph. 
                        By this construction, we get a connected double cover $\Gamma'$ of the graph $\Gamma$. 
                        The graph manifold $M'$ dual to $\Gamma'$ is hence the double cover of $M$. 
                        We denote this covering by $p:M' \to M$. 
                        Then the subgraph $\Gamma'_0$ induced by the vertex set $V'_0$ is a double cover of $\hat{\Gamma}_0$. 
                        The corresponding submanifold of $\Gamma'_0$ is a noncentral Abelian component under the representation $p\circ \rho$. 
                        According to our construction, the fiber of Seifert blocks adjacent to $\Gamma'_0$ are all central types. 
                        Hence, the induced subgraph $\Gamma'_0$ is a maximal noncentral Abelian subgraph.   
                    \end{proof}

                    Let a Seifert block $M_v$ has a fiber with projectively infinite order. 
                    We set $M_v$ contained in a maximal noncentral Abelian components $M_0$ and the corresponding subgraph $\Gamma_0$ is an induced subgraph. 
                    Let $W$ denote the set of all adjacent Seifert blocks of $M_v$. 
                    Then we use $W_A$ to denote the subset of $W$ which contains adjacent Seifert blocks of $M_v$ lying in the maximal noncentral Abelian components $M_0$  
                    and use $W_C$ to denote the rest of elements in $W$ which are not in $M_0$.  
                    Then the equation (\ref{BSF}) can be written in other form, i.e. 
                    \begin{equation}\label{AO}
                        b_v k_v f_v - \sum_{w\in W_A} \frac{b_v}{b_{w,v}} f_w = \sum_{w'\in W_C} \frac{b_v}{b_{w',v}} f_{w'}. 
                    \end{equation}

                    The fiber $f_{w'}$ commute with $f_v$. 
                    Moreover, the right-hand of equation (\ref{AO}) has all fibers with images in the center of $\MG$ under the representation $\rho$ by proposition \ref{ABC}.

                    If a maximal noncentral Abelian component has an induced corresponding subgraph under the $\rep$ $\rho$, then we give a following definition.  
                    \begin{definition}
                        Let $M_0$ be a maximal noncentral Abelian component with an induced corresponding subgraph. 
                        If $M_0$ consists of $n$ Seifert blocks,   
                        then for each Seifert block $M_{0,i}$ of $M_0$, there is an equation (\ref{AO}) associated to $M_{0,i}$. 
                        Then we have $n$ equations for the maximal Abelian component $M_0$. 
                        Denote the equations in form of matrices 
                        \begin{equation*}
                             \bm{e}_0 \bm{f_A} = \bm{f_C}. 
                        \end{equation*}
                        We call $\bm{e}_0$ the associated matrix to the maximal Abelian component $M_0$ under the $\rep$ $\rho$. 
                    \end{definition}
    
                    Let $M_0$ be a maximal noncentral Abelian component with an induced corresponding subgraph. 
                    If $M_0$ consists of $n$ Seifert blocks with fiber denoted by $f_i,1\leq i \leq n$,     
                    we can write down the associated matrix to the maximal noncentral Abelian component $M_0$.  
                    For the $i$-th Seifert block in $M_0$, let $W_C^i$ denote the adjacent Seifert blocks not in $M_0$ ($W_C^i$ may be empty).  
                    Then the matrix equation is as the following: 
                    \begin{equation}
                        \renewcommand\arraystretch{2}
                        \setlength{\arraycolsep}{8pt}
                        \left(
                        \begin{array}{cccc}
                           b_1k_1 & \beta_{12} & \cdots & \beta_{1,n} \\
                           \beta_{21} & b_2k_2 & \cdots & \beta_{2,n} \\
                           \vdots & \vdots & \ddots & \vdots \\
                           \beta_{n,1} & \beta_{n,2} & \cdots & b_nk_n
                        \end{array}
                        \right)
                        \left(
                        \begin{array}{c}
                           f_1 \\ f_2 \\ \vdots \\ f_n
                        \end{array}
                        \right) = 
                        \left(
                        \begin{array}{c}
                           \sum\limits_{w'\in W_C^1} \frac{b_1}{b_{w',1}} f_{w'} \\
                           \sum\limits_{w'\in W_C^2} \frac{b_2}{b_{w',2}} f_{w'} \\
                           \vdots \\
                           \sum\limits_{w'\in W_C^n} \frac{b_n}{b_{w',n}} f_{w'}
                        \end{array}
                        \right),
                    \end{equation}
                    where $\beta_{ij} (i\neq j)$ is $-\frac{b_i}{b_{j,i}}$
                    if the $i$-th Seifert block and the $j$-th Seifert block are adjacent. 
                    Otherwise, $\beta_{ij}$ is 0. 

                    Since the image of $\pi_1(M_0)$ is an Abelian group, we write the group multiplication in the form of the addition. 
                    We can find that the entries of $\bm{e}_0$ are all integers and the diagonal elements of $\bm{e}_0$ are in the form of $b_ik_i$. 
                    Moreover, all elements of the right-hand column are of central type.    

        \section{Virtually nonexistence of faithful representations in closed case}\label{Proof}

            In this section we prove the Theorem \ref{NFR}. 
            In fact, we need to prove that a strictly diagonally dominant closed graph manifold virtually has no vertex faithful $\rep$s to the Seifert motion group. 
        
            \begin{proof}[Proof of Theorem \ref{NFR}]
                By lemma \ref{KLL}, there is a finite cover $M_1$ of the $\gmd$ $M$ 
                such that each Seifert block homeomorphic to a trivial circle bundle over a compact surface with genus greater than 1. 
                Then we can find a JSJ characteristic finite cover $M'$ of $M$ factors through $M_1$ 
                and every Seifert block of $M'$ is also a trivial circle bundle over a compact surface with genus greater than 1. 
                
                Suppose there exists a vertex faithful representation $\rho': \pi_1(M') \to \MG$ which is faithful at vertex $v'$. 
                By corollary \ref{ABF}, the adjacent Seifert blocks of $M'_{v'}$ all have fibers of noncentral type with images of projectively infinite order. 
                Hence, each adjacent block has an Abelian $\rep$ under $\rho'$.    

                Let $M'_0$ be a maximal noncentral Abelian component of $M'$ which contains at least one adjacent block of $M'_{v'}$. 
                Then by proposition \ref{ISG}, there is JSJ characteristic double cover $M''$ 
                such that the preimage $M''_0$ is a maximal noncentral Abelian component whose corresponding subgraph is an induced subgraph. 
                Let $\rho'': \pi_1(M'') \to \MG $ be the representation which is faithful at vertex $v''$ who is one component of the preimage of $v'$.  

                The graph manifold $M''$ is a JSJ characteristic double cover of $M'$, and hence it is a JSJ characteristic finite cover of $M$. 
                Since $M$ is strictly diagonally dominant, by proposition \ref{SDD}, $M''$ is also strictly diagonally dominant. 
                According to the definition \ref{DSB} of a strictly diagonally dominant $\gmd$, each Seifert block of $M''$ is strictly diagonally dominant. 
                Hence, $M''_0$ also satisfies the condition of strictly diagonally dominance. 

                Let $M''_{u''}$ be a Seifert block in $M''_0$.  
                Then $\rho''$ restricted to $\pi_1(M''_{u''})$ is Abelian. 
                We can write an equation (\ref{AO}) for $M''_{u''}$, 
                \[
                    b_{u''} k_{u''} f_{u''} - \sum_{w''_A \in W''_A} \frac{b_{u''}}{b_{w''_A,u''}} f_{w''_A} = \sum_{ w''_C \in W''_C} \frac{b_{u''}}{b_{w''_C,u''}} f_{w''_C},  
                \]
                Here $W''_A$ denotes the set of adjacent blocks of $M''_{u''}$ which lie in $M''_0$ 
                and $W''_C$ denotes the set of other adjacent blocks of $M''_{u''}$ which are not in $M''_0$. 

                Since $M''_{u''}$ is strictly diagonally dominant, we have an inequality
                \[ 
                    |b_{u''} k_{u''}| > \sum_{w''_A \in W''_A} \frac{|b_{u''}|}{|b_{w''_A,u''}|} + \sum_{ w''_C \in W''_C} \frac{|b_{u''}|}{|b_{w''_C,u''}|} 
                                    \geq \sum_{w''_A \in W''_A} \frac{|b_{u''}|}{|b_{w''_A,u''}|}. 
                \]
                Notice that the left-hand is the absolute value of a diagonal element $b_{u''} k_{u''}$ in the associated matrix $\bm{e''}_0$ to the submanifold $M''_0$. 
                The right-hand is the sum of absolute value of other entries which lie in the same row with $b_{u''} k_{u''}$. 
                For each row of the associated matrix $\bm{e''}_0$, the corresponding inequality holds.  
                Then the associated matrix $\bm{e''}_0$ to $M''_0$ is strictly diagonally dominant. Hence, it is invertible. 

                Consider the equation 
                \[ \bm{e''}_0 \bm{f''_A} = \bm{f''_C}. \]
                We have 
                \[
                    (\bm{e''}_0)^*\bm{e''}_0 \bm{f''_A} = (\bm{e''}_0)^*\bm{f''_C}, 
                \] 
                \[
                    \text{det}(\bm{e''}_0)\bm{f''_A} = (\bm{e''}_0)^*\bm{f''_C}, 
                \]
                where $(\bm{e''}_0)^*$ is the adjoint matrix of $\bm{e''}_0$ and $\text{det}(\bm{e''}_0)$ is the determinant of the matrix $\bm{e''}_0$. 
                
                The entries of the matrix $\bm{e''}_0$ are all integers, so are the entries of the adjoint matrix $(\bm{e''}_0)^*$. 
                The images of fibers are all in an Abelian group, the equation has solutions. 
                By remark \ref{ROE}, the elements of any solution all have projectively finite order. 
                This conclusion contradicts the result that the fibers of adjacent blocks of $M''_{v''}$ are noncentral elements with projectively infinite order. 
                So, the vertex faithful $\rep$ $\rho''$ does not exist. 
            
                As for a general finite cover $M_2$, there is a JSJ characteristic finite cover $M_3$ which factors through $M_2$. 
                Then $M_3$ is strictly diagonally dominant. 
                There is no vertex faithful $\rep$s on $M_3$. 
                Hence, there is no vertex faithful $\rep$s on $M_2$. 
                So, a closed strictly diagonally dominant $\gmd$ $M$ virtually has no vertex faithful $\rep$s.
            \end{proof} 

            As a deduced result, a closed strictly diagonally dominant $\gmd$ $M$ virtually has no faithful $\rep$s to the Seifert motion group. 
        
            In another point of view, the proof of Theorem \ref{NFR} is extending a $\rep$ from a Seifert block to the whole $\gmd$. 
            So, we can state in a form of a corollary. 
            \begin{corollary}
                Let $M$ be a strictly diagonally dominant $\gmd$. 
                There is a $\rep$ $\rho$ on a Seifert block $M_v$ is faithful to the Seifert motion group.  
                Then the $\rep$ $\rho$ virtually cannot be extended to the whole $\gmd$ $M$. 
            \end{corollary}        
        
        \section{Existence of a vertex faithful representation in bounded case}\label{EB}

            If a $\gmd$ has torus boundaries, then there is a vertex faithful $\rep$ on it. 
            In other word, the faithful $\rep$ of a Seifert block can be extended to a $\rep$ on the whole $\gmd$. 
            The $\rep$ is not necessarily faithful on the whole $\gmd$. 
            In the following we give a construction for one class of $\gmd$s with nonempty boundary. 
    
            \begin{theorem}
                Let $M$ be a $\gmd$ whose dual graph $\Gamma$ is a tree 
                and Each Seifert block dual to end vertex of the tree has at least one torus boundary. 
                If one Seifert block has a faithful $\rep$ to the Seifert motion group, 
                then the $\rep$ can be extended to the whole $\gmd$ $M$. 
            \end{theorem}
    
            \begin{proof}
                Suppose there is a faithful $\rep$ $\rho_0$ on a Seifert block $M_v$ of the $\gmd$ $M$.  
                Then there is a direction on $\Gamma$ starting from $v$ to all the ends. 

                Let the $\rep$ denoted by $\rho: \pi_1(M) \to \MG$. Let $W$ denote all adjacent blocks of $M_v$. 
                The images of a homological basis $\{f_{w,v},z_{w,v}\}$ of each torus $T_{w,v}$ are determined by a gluing matrix $A_{v,w}$. 
                Then we set the image of $\pi_1(M_w)$ to be Abelian. 
                So the base surface boundaries of $M_w$ have all images in a group $C(\rho(f_w))=\{ h \in \MG \ |\  \rho(f_w) h =h \rho(f_w) \}$.  
                We choose elements from $C(\rho(f_w))$ such that they form a Waldhausen basis for $M_w$. 
                Hence, each adjacent block of $M_v$ has an Abelian $\rep$.   

                Suppose the $\rep$ $\rho$ has been extended to all Seifert blocks with distance less than $n$ to $M_v$. 
                When $n=2$, we have done it in the last paragraph. 
                If a Seifert block $M_u$ with distance $n$ to $M_v$, 
                it is adjacent to a unique Seifert block $M_t$ with distance $n-1$ to $M_v$. 
                The basis $\{ f_{u,t}, z_{u,t} \}$ of the torus $T_{u,t}$ is determined by the gluing matrix $g_{t,u}$ from $T_{t,u}$ to $T_{u,t}$, 
                \[
                    \begin{pmatrix}
                        f_{u,t} \\
                        z_{u,t}
                    \end{pmatrix} =   
                    g_{t,u}
                    \begin{pmatrix}
                        f_{t,u} \\
                        z_{t,u}
                    \end{pmatrix} .  
                \]
                
                Hence, the image of the fiber $f_u$ is determined by the inclusion map. 
                
                Since the graph is a tree with all end points corresponding to Seifert manifolds meeting the boundary of $M$. 
                $M_u,$ has at least one torus boundary is free. 
                Then we can choose elements in the group $C(\rho(f_{u,t}))\cap C(\rho(z_{u,t}))$ to get a Waldhausen basis of $M_u$. 
                Hence, the $\rep$ $\rho$ restricted on $\pi_1(M_u)$ is Abelian. 

                By induction, we can extend the $\rep$ $\rho_0$ to $\rho$ on the whole $\gmd$ such that $\rho$ restricted on $\pi_1(M_v)$ is faithful.
            \end{proof}

    \end{spacing}
 
    \bibliographystyle{alpha}
        \bibliography{Repofgmfd}

\end{document}